\documentclass[12pt,reqno]{amsart}
\usepackage{amsfonts,amsmath,amscd,amssymb,amsthm}
\usepackage{braket,color,enumerate,graphicx,hyperref,mathtools,perpage,url}
\usepackage[margin=2.3cm]{geometry}

\theoremstyle{plain}
\newtheorem{thm}{Theorem}

\newtheorem{lemma}[thm]{Lemma}
\newtheorem{pblm}[thm]{Problem}
\newtheorem{prop}[thm]{Proposition}

\theoremstyle{remark}

\newtheorem*{ex}{\textbf{Example}}
\newtheorem{rmk}[thm]{\textbf{Remark}}
\newtheorem*{rmknn}{\textbf{Remark}}

\numberwithin{equation}{section}

\MakePerPage{footnote} \allowdisplaybreaks 

\newcommand{\B}{\mathbb{B}}
\newcommand{\bs}{\backslash}
\newcommand{\C}{\mathbb C}

\newcommand{\Dc}{\mathcal{D}}

\newcommand{\dimH}{\dim_{\mathrm H}}

\newcommand{\F}{\mathcal F}
\newcommand{\Hb}{\mathbb H}

\newcommand{\Id}{\mathrm{Id}}

\newcommand{\M}{\mathbb M}
\newcommand{\N}{\mathbb N}

\newcommand{\R}{\mathbb{R}}
\newcommand{\s}{\mathbb{S}}

\newcommand{\Z}{\mathbb Z}

\title[Restriction estimates on hyperbolic manifolds]{Tomas-Stein restriction estimates on convex cocompact hyperbolic manifolds. I}

\author{Xiaolong Han}
\email{xiaolong.han@csun.edu}
\address{Department of Mathematics, California State University, Northridge, CA 91330, USA}

\thanks{} 
\date{}

\begin{document}
\maketitle

\begin{abstract}
In this paper, we investigate the Tomas-Stein restriction estimates on convex cocompact hyperbolic manifolds $\Gamma\bs\Hb^{n+1}$. Via the spectral measure of the Laplacian, we prove that the Tomas-Stein restriction estimate holds when the limit set has Hausdorff dimension $\delta_\Gamma<n/2$. This provides an example for which restriction estimate holds in the presence of hyperbolic geodesic trapping.
\end{abstract}

\section{Introduction}
In $\R^d$, the Tomas-Stein restriction theorem \cite{T, St} states that if $1\le p\le p_c:=2(d+1)/(d+3)$, then
\begin{equation}\label{eq:TSRd}
\|R_1f\|_{L^2\left(\s_1^{d-1}\right)}\le A\|f\|_{L^p(\R^d)}\quad\text{for all }f\in C^\infty_0(\R^d),
\end{equation}
where $A>0$ depends only on $d$ and $p$. Here, the Fourier transfer restriction operator (associated with the unit sphere $\s_1^{d-1}$) is defined as
$$R_1f(\xi)=\int_{\R^d}e^{-ix\cdot\xi}f(x)\,dx\quad\text{for }\xi\in\s_1^{d-1}.$$
Let $R_1^\star$ be the adjoint of $R_1$. Since $R_1^\star R_1:L^p(\R^d)\to L^{p'}(\R^d)$ for $p'=p/(p-1)$, the Tomas-Stein restriction estimate \eqref{eq:TSRd} is equivalent to
\begin{equation}\label{eq:RstarRRd}
\left\|R_1^\star R_1\right\|_{L^p(\R^d)\to L^{p'}(\R^d)}\le A^2.
\end{equation}
Denote $\Delta_{\R^d}$ the (positive) Laplacian in $\R^d$. Then $\sqrt{\Delta_{\R^d}}$ has an absolutely continuous spectrum on $[0,\infty)$ and 
$$\sqrt{\Delta_{\R^d}}=\int_0^\infty\lambda\,dE_{\sqrt{\Delta_{\R^d}}}(\lambda),$$
in which $dE_{\sqrt{\Delta_{\R^d}}}$ is the spectral measure of $\sqrt{\Delta_{\R^d}}$. Notice that $dE_{\sqrt{\Delta_{\R^d}}}(\lambda)=R_\lambda^\star R_\lambda$, where $R_\lambda$ is the Fourier restriction operator associated with the sphere $\s_\lambda^{d-1}$ with radius $\lambda$. A direct dilation argument yields
\begin{equation}\label{eq:dilationRd}
\left\|R_\lambda^\star R_\lambda\right\|_{L^p(\R^d)\to L^{p'}(\R^d)}=\lambda^{d\left(\frac1p-\frac{1}{p'}\right)-1}\left\|R_1^\star R_1\right\|_{L^p(\R^d)\to L^{p'}(\R^d)}.
\end{equation} 
Then \eqref{eq:TSRd} and \eqref{eq:RstarRRd} are also equivalent to
\begin{equation}\label{eq:specRd}
\left\|dE_{\sqrt{\Delta_{\R^d}}}(\lambda)\right\|_{L^p(\R^d)\to L^{p'}(\R^d)}\le A^2\lambda^{d\left(\frac1p-\frac{1}{p'}\right)-1}\quad\text{for }1\le p\le p_c.
\end{equation}
The Tomas-Stein restriction problem can therefore be generalized to manifolds $\M$, via spectral measure of $\sqrt{\Delta_\M}$. We assume that the Laplacian $\Delta_\M$ is nonnegative and essentially self-adjoint on $C^\infty_0(\M)\subset L^2(\M)$. (These conditions are automatically true on the convex cocompact hyperbolic manifolds that we consider in this paper. See below for details of the geometric setting.)
\begin{pblm}[Restriction estimates on manifolds via the spectral measure]\label{pblm}
Let $\M$ be a $d$-$\dim$ manifold. Is the following Tomas-Stein restriction estimate true for $\lambda>0$?
\begin{equation}\label{eq:specM}
\left\|dE_{\sqrt{\Delta_\M}}(\lambda)\right\|_{L^p(\M)\to L^{p'}(\M)}\le C\lambda^{d\left(\frac1p-\frac{1}{p'}\right)-1}\quad\text{for }1\le p\le p_c.
\end{equation} 
\end{pblm}
See also the discussion in Chen-Hassell \cite[Section 1.2]{CH}. The parameter $\lambda$ (i.e. energy) here is important since the dilation structure \eqref{eq:dilationRd} in $\R^d$ may not be available on the manifold. We are concerned with whether the restriction estimate \eqref{eq:specM} holds for all $\lambda>0$ on a manifold and how it is influenced by the underlying geometry.

If $\M$ is compact, then the Laplacian $\Delta_\M$ has a discrete spectrum of eigenvalues $0\le\lambda_0^2\le\lambda_1^2\le\cdots\to\infty$ with smooth eigenfunctions $\{u_j\}_{j=0}^\infty$. Formally, $\sqrt{\Delta_\M}=\sum_j\lambda_j\langle u_j,\cdot\rangle u_j$. So the spectral measure $dE_{\sqrt{\Delta_\M}}(\lambda)$ is a sum of Dirac delta measures at $\lambda_j$'s. Therefore, the restriction estimate \eqref{eq:specM} can never hold at $\lambda_j$'s. Instead, the appropriate ``discrete'' version of restriction estimates in this case is for the spectral projection onto finite intervals in the spectrum, e.g. $[\lambda,\lambda+1]$. These estimates in term imply the $L^p$ estimates of spectral clusters. See Sogge \cite[Chapter 5]{So}.

On non-compact and complete manifolds, the restriction estimate \eqref{eq:specM} has been proved in various settings. We mention Guillarmou-Hassell-Sikora \cite{GHS} for asymptotically conic manifolds and Chen-Hassell \cite{CH} for asymptotically hyperbolic manifolds\footnote{See also the recent work of Huang-Sogge \cite{HS}, which includes spectral projection estimates on hyperbolic spaces $\Hb^{n+1}$. The restriction estimates in \eqref{eq:specM} can be derived from \cite[Equation 1.16]{HS}.}, which are the motivation and also main resources for our investigation in the current paper. In both of these two cases, a geodesic non-trapping condition is assumed, that is, there is no geodesic which is contained in some compact region of $\M$; it in particular requires that there are no closed geodesics in $\M$.

Furthermore, Guillarmou-Hassell-Sikora \cite[Section 8C]{GHS} remarked that if there is an elliptic closed geodesic $l\subset\M$, then the restriction estimate \eqref{eq:specM} fails. In this case, one can construct well approximated eigenfunctions (i.e. quasimodes) associated with $l$. See Babich-Lazutkin \cite{BL} and Ralston \cite{R}. Precisely, there are $\lambda_j\to\infty$ and $u_j\in L^2(\M)$ such that 
$$\|(\Delta_\M-\lambda_j^2)u_j\|_{L^2(\M)}\le C_N\lambda_j^{-N}\|u_j\|_{L^2(\M)}\quad\text{for all }N\in\N\text{ as }j\to\infty.$$ 
In fact, the construction of such quasimodes associated with $l$ is local around the geodesic, i.e. $u_j\in L^2(K)$ for some compact $K\supset l$. The existence of these quasimodes ensures that following statement is \textit{invalid} for all $1\le p<2$ and $M>0$ \cite[Proposition 8.7]{GHS}.
$$\exists C>0,\ \exists\lambda_0,\ \forall\lambda\ge\lambda_0,\ \left\|dE_{\sqrt{\Delta_\M}}(\lambda)\right\|_{L^p(\M)\to L^{p'}(\M)}\le C\lambda^M.$$
So the question arises naturally, c.f. \cite[Remark 1.5]{GHS}:
$$\text{Can the restriction estimate \eqref{eq:specM} hold in the presence of non-elliptic closed geodesics?}$$
We focus on hyperbolic closed geodesics in this paper and remark that the (non-)existence of well approximated eigenfunctions as above but associated with a hyperbolic closed geodesic is not completely understood. It is a major problem in the study of Quantum Chaos; see Christianson \cite{Chr} and Zelditch \cite[Section 5]{Z}. Nevertheless, in this paper, we are able to treat the restriction estimate in Problem \ref{pblm} on certain hyperbolic manifolds, where all closed geodesics are hyperbolic. To the author's knowledge, these manifolds are the first examples with geodesic trapping for which the restriction estimate \eqref{eq:specM} holds.

\subsection*{Geometric setting}
Denote $\Hb^{n+1}$ the $(n+1)$-$\dim$ hyperbolic space. Let $\M=\Gamma\bs\Hb^{n+1}$ be a convex cocompact hyperbolic manifold, i.e. $\Gamma$ is a discrete group of orientation preserving isometries of $\Hb^{n+1}$ that consists of hyperbolic elements and $\M$ is geometrically finite and has infinite volume. The set of closed geodesics in $\M$ corresponds to the conjugacy classes within the group $\Gamma$. 

The size of the geodesic trapped set is characterized by the limit set $\Lambda_\Gamma$ of $\Gamma$. The limit set $\Lambda_\Gamma\subset\partial\Hb^{n+1}$ is the set of accumulation points on the orbits $\Gamma z$, $z\in\Hb^{n+1}$. The Hausdorff dimension of $\Lambda_\Gamma$, $\delta_\Gamma:=\dimH\Lambda_\Gamma\in[0,n)$. Then the trapped set of the geodesic flow in the unit tangent bundle $S\M$ has Hausdorff dimension $2\delta_\Gamma+1$. See Patterson \cite{P} and Sullivan \cite{Su}.

\begin{ex}
The simplest example of convex cocompact hyperbolic manifolds is the hyperbolic cylinder $\Gamma\bs\Hb^{n+1}$, in which $\Gamma=\Z$ acts on $\Hb^{n+1}$ by powers of a fixed dilation. In this case, the limit set $\Lambda_\Gamma=\{0,\infty\}$. There is only one closed geodesic. On non-elementary convex cocompact hyperbolic manifolds, however, there can be infinitely many closed geodesics.
\end{ex}

It is now well-known by Lax-Phillips \cite{LP1, LP2} that the spectrum of Laplacian $\Delta_\M$ consists of at most finitely many eigenvalues in the interval $(0,n^2/4)$ and absolutely continuous spectrum $[n^2/4,\infty)$ with no embedded eigenvalues. It is hence convenient in notation to consider the restriction estimates for the operator
\begin{equation}\label{eq:P}
P_\M=\left(\Delta_\M-\frac{n^2}{4}\right)_+^\frac12,
\end{equation}
where $(\cdot)_+=\max\{\cdot,0\}$. The operator $P_\M$ has an absolutely continuous spectrum $[0,\infty)$. 

Before we state the main theorem, we remark that the range $1\le p\le p_c$ in the restriction estimate \eqref{eq:specM} can be extended to $1\le p<2$ if $\M=\Hb^{n+1}$ (more generally, $\M$ is a non-trapping asymptotically hyperbolic manifold, see Chen-Hassell \cite[Theorem 1.6 and Remark 1.7]{CH}.) This range is larger than the one on $\R^d$ in \eqref{eq:specRd} and is related to the Kunze-Stein theory \cite{KS} of harmonic analysis on semisimple Lie groups. The extended range of $p$ for restriction estimate persists on the hyperbolic manifolds considered here.

Our main theorem states
\begin{thm}[Restriction estimates on convex cocompact hyperbolic manifolds]\label{thm:specM}
Let $\M=\Gamma\bs\Hb^{n+1}$ be a convex cocompact hyperbolic manifold for which $\delta_\Gamma<n/2$. Then there exists $C>0$ depending on $\M$ and $p$ such that at high energy $\lambda\ge1$,
$$\left\|dE_{P_\M}(\lambda)\right\|_{L^p(\M)\to L^{p'}(\M)}\le 
\begin{cases}
C\lambda^{(n+1)\left(\frac1p-\frac{1}{p'}\right)-1} & \text{for }1\le p\le p_c=\frac{2(n+2)}{n+4},\\
C\lambda^{n\left(\frac1p-\frac12\right)} & \text{for }p_c\le p<2.
\end{cases}$$
\end{thm}

Some remarks on the proof of the theorem and further investigations are in order.

\begin{rmk}[Restriction estimates at low energy]\label{rmk:low}
Under the condition  in Theorem \ref{thm:specM}, the resolvent (acting on appropriate spaces, see e.g. Bourgain-Dyatlov \cite{BD})
$$\mathcal R_\lambda:=(\Delta_\M-n^2/4-\lambda^2)^{-1}$$ 
is holomorphic in the half complex plane $\{\lambda\in\C:\mathrm{Im}\lambda>-(n/2-\delta_\Gamma)\}$ by the Patterson-Sullivan theory \cite{P, Su}. So in this half plane, there are no resonances, which are the poles of $\mathcal R_\lambda$ in $\C$. (That is, there is a spectral gap of size at least $n/2-\delta_\Gamma>0$.) In particular, there is no resonance at the bottom of the continuous spectrum $[0,\infty)$ of $\Delta_\M-n^2/4$. This condition guarantees that the restriction estimates at low energy $\lambda\le1$ in Chen-Hassell \cite[Theorems 1.5 and 1.6]{CH} remain valid. That is, at low energy $\lambda\le1$,
$$\left\|dE_{P_\M}(\lambda)\right\|_{L^p(\M)\to L^{p'}(\M)}\le C\lambda^2\quad\text{for }1\le p<2.$$
\end{rmk}

\begin{rmk}[Critical $\delta_\Gamma$ for the restriction estimate]
Our method in this paper can not treat the restriction estimate in Problem \ref{pblm} on $\M=\Gamma\bs\Hb^{n+1}$ for which $\delta_\Gamma\ge n/2$. It is not yet clear whether the restriction estimate \eqref{eq:specM} holds on such manifolds with large limit sets (and thus with large hyperbolic trapped sets). Notice that in the extreme case when $\M$ is compact, $\Lambda_\Gamma=\partial\Hb^{n+1}$ (so $\delta_\Gamma=n$) and \eqref{eq:specM} fails. It is interesting to find the ``critical'' dimension $n/2\le\delta_c\le n$ of the limit sets for which \eqref{eq:specM} fails for the corresponding hyperbolic manifolds. We plan to investigate this problem in a future work. Some relevent spectral information on hyperbolic surfaces (i.e. $\dim\M=2$) when $\delta_\Gamma\ge1/2$ has recently been proved, in particular, Bourgain-Dyatlov \cite{BD} established an essential spectral gap for the resolvent $\mathcal R_\lambda$ in $\C$. 
\end{rmk}

\begin{rmk}[More general geometries for which the hyperbolic trapped sets are small]
The proof of Theorem \ref{thm:specM} is inspired by Burq-Guillarmou-Hassell \cite[Theorem 1.1]{BGH}, in which they studied the Strichartz estimates for Schr\"odinger equation on the convex cocompact hyperbolic manifolds for which $\delta_\Gamma<n/2$. In the same paper, the authors also treated more general classes of manifolds, including manifolds that contain small sets of hyperbolic trapped sets but \textit{not} necessarily with constant negative curvature. Instead of using the Hausdorff dimension of the limit set to characterize the size of trapped set, they used the topological pressure condition\footnote{The topological pressure condition reduces to the condition about Hausdorff dimension of the limit set if the manifold has constant negative curvature. See \cite[Lemma 3.5]{BGH}.}. It is interesting to see if Theorem \ref{thm:specM} can be generalized to such setting. 
\end{rmk}

\section{Proof of Theorems \ref{thm:specM}}

The main tool to prove the Tomas-Stein restriction estimates in Theorem \ref{thm:specM} is the abstract spectral theory by Guillarmou-Hassell-Sikora \cite[Theorem 3.1]{GHS}. See also Chen \cite{Che}. 

\begin{thm}\label{thm:specabs}
Let $(X,d,\mu)$ be a metric measure space and $L$ be an abstract nonnegative self-adjoint operator on $L^2(X,\mu)$. Assume that the spectral measure $dE_{\sqrt L}(\lambda)$ has a Schwartz kernel $dE_{\sqrt L}(\lambda)$ for $x,y\in X$. Suppose that there is a subset $I\subset[0,\infty)$ such that for $\lambda\in I$,
\begin{equation}\label{eq:specpt}
\left|\frac{d^j}{d\lambda^j}dE_{\sqrt L}(\lambda)(x,y)\right|\le C\lambda^{m-1-j}\left(1+\lambda d(x,y)\right)^{-(m-1)/2+j},
\end{equation}
in which
\begin{enumerate}[(i).]
\item $j=0$, $j=m/2-1$, and $j=m/2$ if $m$ is even,
\item $j=m/2-3/2$ and $j=m/2+1/2$ if $m$ is odd.
\end{enumerate}
Then the following Tomas-Stein restriction estimate holds for all $\lambda\in I$ and $1\le p\le p_c$.
$$\left\|dE_{\sqrt L}(\lambda)\right\|_{L^p(\M)\to L^{p'}(\M)}\le C\lambda^{m\left(\frac1p-\frac{1}{p'}\right)-1}.$$
\end{thm}

In application, we substitute $L=(\Delta_\M-n^2/4)_+$ into the above theorem to prove Theorem \ref{thm:specM}. We begin from the estimates on the hyperbolic space $\Hb^{n+1}$. For notational simplicity, from now on we denote
$$\Hb=\Hb^{n+1}.$$
Let
$$P_\Hb=\left(\Delta_\Hb-\frac{n^2}{4}\right)_+^\frac12.$$
The following pointwise estimates are from Chen-Hassell \cite[Equations (1.9) and (1.10)]{CH}. They actually proved these estimates on asymptotically hyperbolic manifolds with geodesic non-trapping condition.
\begin{prop}[Pointwise estimates of the spectral measure on $\Hb$]\label{prop:specptH}
The Schwartz kernel of $dE_{P_\Hb}(\lambda)(x,y)$ for $\lambda\ge1$ and $x,y\in\Hb(=\Hb^{n+1})$ satisfies 
$$\left|\frac{d^j}{d\lambda^j}dE_{P_\Hb}(\lambda)(x,y)\right|\le\begin{cases}
C\lambda^{n-j}(1+\lambda d_\Hb(x,y))^{-n/2+j} & \text{for }d_\Hb(x,y)\le1;\\
C\lambda^{n/2}d_\Hb(x,y)^je^{-nd_\Hb(x,y)/2} & \text{for }d_\Hb(x,y)\ge1.
\end{cases}$$
Here, $d_\Hb$ is the hyperbolic distance in $\Hb$.
\end{prop}

\begin{rmknn}
The pointwise upper bounds in the two distance ranges above reflect two different behaviors of the spectral measure on hyperbolic spaces. 
\begin{enumerate}[(1).]
\item When $d_\Hb(x,y)\le1$, the estimate is similar to the one in $\R^d$:
\begin{equation}\label{eq:specptRd}
\frac{d^j}{d\lambda^j}dE_{\sqrt{\Delta_{\R^d}}}(\lambda)(x,y)=\frac{d^j}{d\lambda^j}\int_{\s_\lambda^{d-1}}e^{i(x-y)\cdot\xi}\,d\xi\sim\lambda^{d-1-j}(1+\lambda d_{\R^d}(x,y))^{-\frac{d-1}{2}+j},
\end{equation}
following the standard non-stationary phase asymptotics. See e.g. Stein \cite{St}. 
\item When $d_\Hb(x,y)\ge1$, the exponential estimate is different with the one in $\R^d$ and is related to the exponential volume growth in radii of geodesic balls in the hyperbolic space.
\end{enumerate}
\end{rmknn}

\begin{rmk}\label{rmk:specptH}
In Proposition \ref{prop:specptH}, the distance range cutoff at $d_\Hb(x,y)=1$ is rather arbitrary. Give a convex cocompact group $\Gamma$. For each $\gamma\in\Gamma$, there is a unique hyperbolic line in $\Hb$, called the axis of $\gamma$, which is invariant under $\gamma^k$, $k\in\N$. Then $l_\gamma:=d(z,\gamma z)$ for all $z$ on the axis and is called the displacement length of $\gamma$. Moreover, $l_\gamma=\min_{z\in\Hb}d_\Hb(z,\gamma z)$. Denote
\begin{equation}\label{eq:shortest}
l_0=\min_{\gamma\in\Gamma\setminus\{\Id\}}\{l_\gamma\}.
\end{equation}
We know that $l_0>0$ since $\Gamma$ is a discrete group. In the following, we instead use the spectral measure pointwise estimates on $\Hb$ at the distance range cutoff $d_\Hb(x,y)=l_0/2$:
\begin{equation}\label{eq:specptH}
\left|\frac{d^j}{d\lambda^j}dE_{P_\Hb}(\lambda)(x,y)\right|\le\begin{cases}
C\lambda^{n-j}(1+\lambda d_\Hb(x,y))^{-n/2+j} & \text{for }d_\Hb(x,y)<l_0/2;\\
C\lambda^{n/2}d_\Hb(x,y)^je^{-nd_\Hb(x,y)/2} & \text{for }d_\Hb(x,y)\ge l_0/2.
\end{cases}
\end{equation}
But of course now the constant $C$ depends on $l_0$ (therefore on $\Gamma$).
\end{rmk}

Let $\F\subset\Hb$ be a fundamental domain of $\M=\Gamma\bs\Hb$. Then for $x,y\in\F$,
\begin{equation}\label{eq:specMGamma}
dE_{P_\M}(\lambda)(x,y)=\sum_{\gamma\in\Gamma}dE_{P_\Hb}(\lambda)(x,\gamma y).
\end{equation}

\begin{rmknn}[Spectral measure on Euclidean cylinders]
We remark that the convex cocompact group structure of $\Gamma$ on $\Hb$ is crucial for the restriction estimates on $\Gamma\bs\Hb$. For example, take $\M=\Gamma\bs\R^2$ as a Euclidean cylinder. Here, $\Gamma=\Z$ acts on $\R^2$ by powers of a fixed translation $x\to x+l$, $l\in\R^2\setminus\{0\}$. Then by \eqref{eq:specptRd},
\begin{eqnarray*}
\frac{d}{d\lambda}dE_{\sqrt{\Delta_\M}}(\lambda)(x,y)&=&\sum_{k\in\Z}dE_{\sqrt{\Delta_{\R^d}}}(\lambda)(x,y+kl)\\
&\sim&\sum_{k\in\Z}(1+\lambda d_{\R^d}(x,y+kl))^\frac12\\
&\gtrsim&\lambda^\frac12|l|^\frac12\sum_{k\in\Z}|k|^\frac12
\end{eqnarray*}
clearly fails the estimate in Theorem \ref{thm:specabs} when $m=2$ and $j=1$. On the other hand, there are elliptic closed geodesics $\{(x+tl):t\in[0,1)\}$ and by Guillarmou-Hassell-Sikora \cite[Section 8C]{GHS} the restriction estimate \eqref{eq:specM} fails.
\end{rmknn}

We control the summation in the right-hand-side of \eqref{eq:specMGamma} by the Patterson-Sullivan theory \cite{P, Su}. In particular, the Patterson-Sullivan theory concludes that the Poincar\'e series
\begin{equation}\label{eq:Poincare}
G_s(x,y):=\sum_{\gamma\in\Gamma}e^{-sd_\Hb(x,\gamma y)}
\end{equation}
is convergent if and only if $s>\delta_\Gamma$. In fact, by the triangle inequalities 
$$d_\Hb(y,\gamma y)-d_\Hb(x,y)\le d(x,\gamma y)\le d_\Hb(x,y)+d_\Hb(y,\gamma y),$$
we have that
$$e^{-sd_\Hb(x,y)}e^{-sd_\Hb(y,\gamma y)}\le e^{-sd_\Hb(x,\gamma y)}\le e^{sd_\Hb(x,y)}e^{-sd_\Hb(y,\gamma y)}.$$
Summing over $\gamma\in\Gamma$,
$$e^{-sd_\Hb(x,y)}G_s(y,y)\le G_s(x,y)\le e^{sd_\Hb(x,y)}G_s(y,y).$$
So the convergence of the Poincar\'e series \eqref{eq:Poincare} is independent of $x$ and $y$. When the series $G_s(x,y)$ converges, that is, $s<\delta_\Gamma$, we need a quantitative estimate of it that is sufficient for our purpose. 

Following Borthwick \cite[Section 2.5.2]{B}, if $s<\delta_\Gamma$, then
\begin{equation}\label{eq:Poincarel}
\sum_{\gamma\in\Gamma}e^{-sl_\gamma}<C_s,
\end{equation}
in which $C_s$ depends on $s$ and $\Gamma$.  It immediately follows that for all $R>0$,
\begin{equation}\label{eq:Nl}
N(R):=\#\{\gamma\in\Gamma:l_\gamma\le R\}\le C_R,
\end{equation}
in which $C_R$ depends on $R$ and $\Gamma$.

\begin{lemma}\label{lemma:dist}
Let $\F$ be a fundamental domain of $\M=\Gamma\bs\Hb$. There are constants $R,C>1$ such that for all $\gamma\in\Gamma$ with $l_\gamma>R$ and any $k\in\N$, we have that
$$e^{-d_\Hb(x,\gamma y)}\le Ce^{-l_\gamma}\min\{1,d_\F(x,y)^{-k}\}\quad\text{for all }x,y\in\F.$$
Here, $d_\F(x,y)$ is the distance between $x$ and $y$ in $\F$.
\end{lemma}
\begin{proof}
We use the Poincar\'e ball model $\B$ of the hyperbolic space $\Hb$ and denote $|z|$ the Euclidean norm of $z\in\B$. From Guillarmou-Moroianu-Park \cite[Lemma 5.2]{GMP}, there are positive constants $R$ and $C$ such that for all $\gamma\in\Gamma$ with $l_\gamma>R$ and all $x,y\in\F$,
$$e^{-d_\Hb(x,\gamma y)}\le Ce^{-l_\gamma}(1-|x|^2)(1-|y|^2)\le Ce^{-l_\gamma}.$$
Notice that $d_\Hb(x,\gamma y)=d_\Hb(\gamma_0x,\gamma_0\gamma y)$ for any hyperbolic isometry $\gamma_0$ of $\B$. Choose $\gamma_0$ such that $\gamma_0z=e$, where $e$ is the origin in $\B$. Therefore without loss of generality, we can assume that $\F\ni e$ and $x=e$. Note that 
$$d_\F(e,y)=\log\left(\frac{1+|y|}{1-|y|}\right)\le C\log\left(\frac{1}{1-|y|}\right).$$
It thus follows that for all $k\in\N$,
$$(1-|e|^2)(1-|y|^2)=1-|y|^2\le C\left[\log\left(\frac{1}{1-|y|}\right)\right]^{-k}\le Cd_\F(e,y)^{-k}.$$
Hence,
$$e^{-d_\Hb(e,\gamma y)}\le Ce^{-l_\gamma}d_\F(e,y)^{-k}.$$
\end{proof}

\begin{rmk}
Before proving Theorem \ref{thm:specM}, we remark that Chen-Ouhabaz-Sikora-Yan \cite{COSY} developed an abstract system that includes some characterization of the restriction estimates by certain dispersive estimates. In particular, by \cite[Section II.2]{COSY}\footnote{Majority of \cite{COSY} requires that the geometry satisfies volume doubling condition, which the hyperbolic manifolds clearly do not. However, the results in \cite[Section II.2]{COSY} are valid on all metric spaces.}, one can deduce the restriction estimates \eqref{eq:specM} in certain range of $p$ from the dispersive estimate
$$\|e^{it\Delta_\M}\|_{L^1(\M)\to L^\infty(\M)}\le C|t|^{-k}\quad\text{fro some }k>0.$$
However, as seen in Burq-Guillarmou-Hassell \cite[Theorem 1.1]{BGH}, such dispersive estimate on hyperbolic manifolds in general is not sufficient to imply the restriction estimates in the range of $1\le p\le p_c$. On a manifold, the relations between spectral measure estimates in \eqref{eq:specM}, dispersive estimates, and also Strichartz estimates for Schr\"odinger equation are not yet clear. See Burq-Guillarmou-Hassell \cite[Remark 1.3]{BGH}.
\end{rmk}

We now proceed to prove Theorem \ref{thm:specM} by Theorem \ref{thm:specabs}. Fix $x,y\in\M=\Gamma\bs\Hb$, we choose the Dirichlet domain of the point $y$ for the representation of $\M$:
$$\Dc=\Dc_y:=\{z\in\Hb:d_\Hb(z,y)<d_\Hb(z,\gamma y)\text{ for all }\gamma\in\Gamma\setminus\{\Id\}\}.$$
Also, the distance of $x,y$ in $\Dc$ equals $d_\Hb(x,y)$.

To estimate the summation in \eqref{eq:specMGamma}, we first take $\gamma=\Id$.  
\begin{itemize}
\item[Case I.] $d_\Hb(x,y)<l_0/2$. Then the spectral measure pointwise estimate on $\Hb$ in the first distance range of \eqref{eq:specptH} applies. But it coincides with \eqref{eq:specpt} in Theorem \ref{thm:specabs}. 

\item[Case II.] $d_\Hb(x,y)\ge l_0/2$. Then the spectral measure pointwise estimate on $\Hb$ in the second distance range of \eqref{eq:specptH} applies. It is straightforward to see that
\begin{eqnarray*}
\left|\frac{d^j}{d\lambda^j}dE_{P_\Hb}(\lambda)(x,y)\right|&\le&C\lambda^{n/2}d_\Hb(x,y)^je^{-nd_\Hb(x,y)/2}\\
&\le&C\lambda^{n-j}\left(1+\lambda d_\Hb(x,y)\right)^{-n/2+j}.
\end{eqnarray*}
\end{itemize}
In both of these cases for $\gamma=\Id$, the corresponding term $dE_{P_\Hb}(\lambda)(x,\gamma y)$ in the summation \eqref{eq:specMGamma} satisfies the condition \eqref{eq:specpt} in Theorem \ref{thm:specabs}. We then discuss $\gamma\in\Gamma\setminus\{\Id\}$, in the same two cases as above.
\begin{itemize}
\item[Case I.] $d_\Hb(x,y)<l_0/2$. Then $d_\Hb(x,\gamma y)\ge l_0/2$ for all $\gamma\in\Gamma\setminus\{\Id\}$. If not, i.e. $d_\Hb(x,\gamma y)<l_0/2$, then triangle inequality implies that
$$l_\gamma=\min_{z\in\Hb}d_\Hb(z,\gamma z)\le d_\Hb(y,\gamma y)\le d_\Hb(x,y)+d_\Hb(x,\gamma y)<l_0.$$
contradicting with the fact that $l_0=\min_{\gamma\in\Gamma\setminus\{\Id\}}\{l_\gamma\}$ defined in \eqref{eq:shortest}. 

\item[Case II.] $d_\Hb(x,y)\ge l_0/2$. Then by the definition of the Dirichlet domain, 
$$d_\Hb(x,\gamma y)>d_\Hb(x,y)\ge l_0/2\quad\text{for all }\gamma\in\Gamma\setminus\{\Id\}.$$ 
\end{itemize}

Up to this point, to estimate the summation in \eqref{eq:specMGamma}, we only need to estimate the terms for $\gamma\in\Gamma\setminus\{\Id\}$. Moreover, in \eqref{eq:specptH}, the spectral measure pointwise estimate in second distance range applies only.

We write the proof for the restriction estimates \eqref{eq:specM} when $\dim\M=n+1$ is even (so $n$ is odd). To this end, we verify \eqref{eq:specpt} with $\sqrt L=P_\M$ for $j=0$, $j=(n-1)/2$, and $j=(n+1)/2$. The case when $\dim\M=n+1$ is odd proceeds with little modification. 

Write
\begin{eqnarray*}
&&\sum_{\gamma\in\Gamma\setminus\{\Id\}}dE_{P_\Hb}(\lambda)(x,\gamma y)\\
&=&\sum_{\gamma\in\Gamma\setminus\{\Id\}:l_\gamma\le R}dE_{P_\Hb}(\lambda)(x,\gamma y)+\sum_{\gamma\in\Gamma\setminus\{\Id\}:l_\gamma>R}dE_{P_\Hb}(\lambda)(x,\gamma y),
\end{eqnarray*}
in which $R$ is from Lemma \ref{lemma:dist}.

\subsection{The estimate for $j=0$}
Using the fact that $e^{-st}\le Ct^{-k}$ for any $k\in\R$ uniformly on $t\in(0,\infty)$, 
\begin{eqnarray*}
\sum_{\gamma\in\Gamma\setminus\{\Id\}:l_\gamma\le R}\left|dE_{P_\Hb}(\lambda)(x,\gamma y)\right|&\le&C\sum_{\gamma\in\Gamma\setminus\{\Id\}:l_\gamma\le R}\lambda^{\frac n2}e^{-\frac n2d_\Hb(x,\gamma y)}\\
&\le&CN(R)\lambda^{\frac n2}e^{-\frac n2d_\Hb(x,y)}\\
&\le&C\lambda^n(1+\lambda d_\Hb(x,y))^{-\frac n2}.
\end{eqnarray*}
Here, $C$ depends on $R$.

Set $s$ such that $0<\delta_\Gamma<s<n/2$. Since $l_\gamma>R$, we apply Lemma \ref{lemma:dist} to compute that
\begin{eqnarray*}
\sum_{\gamma\in\Gamma\setminus\{\Id\}:l_\gamma>R}\left|dE_{P_\Hb}(\lambda)(x,\gamma y)\right|&\le&C\sum_{\gamma\in\Gamma:l_\gamma>R}\lambda^{\frac n2}e^{-\frac n2d_\Hb(x,\gamma y)}\\
&\le&C\lambda^{\frac n2}\sum_{\gamma\in\Gamma:l_\gamma>R}e^{-sd_\Hb(x,\gamma y)}\\
&\le&C\lambda^{\frac n2}d_\Hb(x,y)^{-k}\sum_{\gamma\in\Gamma}e^{-sl_\gamma}\\
&\le&C\lambda^n(1+\lambda d_\Hb(x,y))^{-\frac n2},
\end{eqnarray*}
by choosing $k$ large enough. Here, we used \eqref{eq:Nl} so the constant $C$ here depends on $s$ and $R$.

The above two estimates together imply that
$$\sum_{\gamma\in\Gamma\setminus\{\Id\}}\left|dE_{P_\Hb}(\lambda)(x,\gamma y)\right|\le C\lambda^n(1+\lambda d_\Hb(x,y))^{-\frac n2}.$$

\subsection{The estimate for $j=(n-1)/2$}
Set $s$ such that $0<\delta_\Gamma<s<n/2$. First we have that
\begin{eqnarray*}
\sum_{\gamma\in\Gamma\setminus\{\Id\}:l_\gamma\le R}\left|\frac{d^{(n-1)/2}}{d\lambda^{(n-1)/2}}dE_{P_\Hb}(\lambda)(x,\gamma y)\right|&\le&C\sum_{\gamma\in\Gamma\setminus\{\Id\}:l_\gamma\le R}\lambda^{\frac n2}d_\Hb(x,\gamma y)^\frac{n-1}{2}e^{-\frac n2d_\Hb(x,\gamma y)}\\
&\le&C\lambda^{\frac n2}\sum_{\gamma:l_\gamma\le R}e^{-sd_\Hb(x,\gamma y)}\\
&\le&CN(R)\lambda^{\frac n2}e^{-sd_\Hb(x,y)}\\
&\le&C\lambda^\frac{n+1}{2}(1+\lambda d_\Hb(x,y))^{-\frac12}.
\end{eqnarray*}
Then for $l_\gamma>R$, we apply Lemma \ref{lemma:dist} to compute that
\begin{eqnarray*}
\sum_{\gamma\in\Gamma\setminus\{\Id\}:l_\gamma>R}\left|\frac{d^{(n-1)/2}}{d\lambda^{(n-1)/2}}dE_{P_\Hb}(\lambda)(x,\gamma y)\right|&\le&C\sum_{\gamma\in\Gamma:l_\gamma>R}\lambda^{\frac n2}d_\Hb(x,\gamma y)^\frac{n-1}{2}e^{-\frac n2d_\Hb(x,\gamma y)}\\
&\le&C\lambda^{\frac n2}\sum_{\gamma\in\Gamma:l_\gamma>R}e^{-sd_\Hb(x,\gamma y)}\\
&\le&C\lambda^{\frac n2}d_\Hb(x,y)^{-k}\sum_{\gamma\in\Gamma}e^{-sl_\gamma}\\
&\le&C\lambda^\frac{n+1}{2}(1+\lambda d_\Hb(x,y))^{-\frac12},
\end{eqnarray*}
by choosing $k$ large enough. 

The above two estimates together imply that 
$$\sum_{\gamma\in\Gamma\setminus\{\Id\}}\left|\frac{d^{(n-1)/2}}{d\lambda^{(n-1)/2}}dE_{P_\Hb}(\lambda)(x,\gamma y)\right|\le C\lambda^\frac{n+1}{2}(1+\lambda d_\Hb(x,y))^{-\frac12}.$$

\subsection{The estimate for $j=(n+1)/2$}
Set $s$ such that $0<\delta_\Gamma<s<n/2$. First similarly as in the above subsection we have that
$$\sum_{\gamma\in\Gamma\setminus\{\Id\}:l_\gamma\le R}\left|\frac{d^{(n+1)/2}}{d\lambda^{(n+1)/2}}dE_{P_\Hb}(\lambda)(x,\gamma y)\right|\le C\lambda^\frac{n-1}{2}(1+\lambda d_\Hb(x,y))^\frac12.$$
Then for $l_\gamma>R$, we apply Lemma \ref{lemma:dist} to compute that
\begin{eqnarray*}
\sum_{\gamma\in\Gamma\setminus\{\Id\}:l_\gamma>R}\left|\frac{d^{(n+1)/2}}{d\lambda^{(n+1)/2}}dE_{P_\Hb}(\lambda)(x,\gamma y)\right|&\le&C\sum_{\gamma\in\Gamma:l_\gamma>R}\lambda^{\frac n2}d_\Hb(x,\gamma y)^\frac{n+1}{2}e^{-\frac n2d_\Hb(x,\gamma y)}\\
&\le&C\lambda^{\frac n2}\sum_{\gamma\in\Gamma:l_\gamma>R}e^{-sd_\Hb(x,\gamma y)}\\
&\le&C\lambda^{\frac n2}d_\Hb(x,y)^{-k}\sum_{\gamma\in\Gamma}e^{-sl_\gamma}\\
&\le&C\lambda^\frac{n-1}{2}(1+\lambda d_\Hb(x,y))^\frac12.
\end{eqnarray*}
The above two estimates together implies that
$$\sum_{\gamma\in\Gamma\setminus\{\Id\}}\left|\frac{d^{(n+1)/2}}{d\lambda^{(n+1)/2}}dE_{P_\Hb}(\lambda)(x,\gamma y)\right|\le C\lambda^\frac{n-1}{2}(1+\lambda d_\Hb(x,y))^\frac12.$$

By the abstract theory of restriction estimates in Theorem \ref{thm:specabs}, Theorem \ref{thm:specM} for the range $1\le p\le p_c$ follows in even dimensions. The case for odd dimension is similar and we omit it here.

\subsection{Proof of Theorem \ref{thm:specM} for $p_c\le p<2$}
We argue the restriction estimates in the range $p_c\le p<2$ similarly as in Chen-Hassell \cite[Section 2.2]{CH}, i.e. Theorem \ref{thm:specM} for $p_c\le p<2$ follows 
$$\left|\frac{d^j}{d\lambda^j}dE_{P_\M}(\lambda)(x,y)\right|\le C\lambda^\frac n2\quad\text{for all }j\ge1.$$
Again we only need to estimate the summation in \eqref{eq:specMGamma} for $\gamma\ne\Id$ and apply the spectral measure pointwise estimate \eqref{eq:specptH} in second distance range. Set $s$ such that $0<\delta_\Gamma<s<n/2$. First we have that
\begin{eqnarray*}
\sum_{\gamma\in\Gamma\setminus\{\Id\}:l_\gamma\le R}\left|\frac{d^j}{d\lambda^j}dE_{P_\Hb}(\lambda)(x,\gamma y)\right|&\le&C\sum_{\gamma\in\Gamma\setminus\{\Id\}:l_\gamma\le R}\lambda^{\frac n2}d_\Hb(x,\gamma y)^je^{-\frac n2d_\Hb(x,\gamma y)}\\
&\le&C\lambda^{\frac n2}\sum_{\gamma:l_\gamma\le R}e^{-sd_\Hb(x,\gamma y)}\\
&\le&CN(R)\lambda^{\frac n2}\\
&\le&C\lambda^\frac n2.
\end{eqnarray*}
Then for $l_\gamma>R$, we apply $e^{-d_\Hb(x,\gamma y)}\le Ce^{-l_\gamma}$ from Lemma \ref{lemma:dist} to compute that
\begin{eqnarray*}
\sum_{\gamma\in\Gamma\setminus\{\Id\}:l_\gamma>R}\left|\frac{d^j}{d\lambda^j}dE_{P_\Hb}(\lambda)(x,\gamma y)\right|&\le&C\sum_{\gamma\in\Gamma:l_\gamma>R}\lambda^{\frac n2}d_\Hb(x,\gamma y)^je^{-\frac n2d_\Hb(x,\gamma y)}\\
&\le&C\lambda^{\frac n2}\sum_{\gamma\in\Gamma:l_\gamma>R}e^{-sd_\Hb(x,\gamma y)}\\
&\le&C\lambda^{\frac n2}\sum_{\gamma\in\Gamma}e^{-sl_\gamma}\\
&\le&C\lambda^\frac n2.
\end{eqnarray*}

The above two estimates together imply that 
$$\sum_{\gamma\in\Gamma\setminus\{\Id\}}\left|\frac{d^j}{d\lambda^j}dE_{P_\Hb}(\lambda)(x,\gamma y)\right|\le C\lambda^\frac n2.$$

\section*{Acknowledgment}
The author benefited from the conversations with Zihua Guo and Andrew Hassell.


\begin{thebibliography}{99}

\bibitem[BL]{BL} V. M. Babich and V. F. Lazutkin, 
\textit{Eigenfunctions concentrated near a closed geodesic}, pp. 9--18 in Topics in Mathematical Physics, vol. 2, edited by M. S. Birman, Consultant's Bureau, New York, 1968.

\bibitem[B]{B} D. Borthwick, 
\textit{Spectral theory of infinite-area hyperbolic surfaces}. Second edition. Birkh\"auser/Springer, 2016. 

\bibitem[BD]{BD} J. Bourgain and S. Dyatlov,
\textit{Spectral gaps without the pressure condition}. Ann. of Math. (2) 187 (2018), no. 3, 825--867.


\bibitem[BGH]{BGH} N. Burq, C. Guillarmou, and A. Hassell, 
\textit{Strichartz estimates without loss on manifolds with hyperbolic trapped geodesics}, Geom. Funct. Anal. 20 (2010), no. 3, 627--656. 

\bibitem[Che]{Che} X. Chen,
\textit{Stein-Tomas restriction theorem via spectral measure on metric measure spaces}. Math. Z. 289 (2018), no. 3-4, 829--835.

\bibitem[Chr]{Chr} H. Christianson, 
\textit{Semiclassical non-concentration near hyperbolic orbits}. J. Funct. Anal. 246 (2007), no. 2, 145--195.

\bibitem[CH]{CH} X. Chen and A. Hassell,
\textit{Resolvent and Spectral Measure on Non-Trapping Asymptotically Hyperbolic Manifolds II: Spectral Measure, Restriction Theorem, Spectral Multipliers}. Ann. Inst. Fourier (Grenoble) 68 (2018), no. 3, 1011--1075.

\bibitem[COSY]{COSY} P. Chen, E. M. Ouhabaz, A. Sikora, and L. Yan,
\textit{Restriction estimates, sharp spectral multipliers and endpoint estimates for Bochner-Riesz means}. J. Anal. Math. 129 (2016), 219--283.

\bibitem[GHS]{GHS} C. Guillarmou, A. Hassell, and A. Sikora, 
\textit{Restriction and spectral multiplier theorems on asymptotically conic manifolds}. Anal. PDE 6 (2013), no. 4, 893--950.

\bibitem[GMP]{GMP} C. Guillarmou, S. Moroianu, and J. Park, 
\textit{Eta invariant and Selberg zeta function of odd type over convex co-compact hyperbolic manifolds}. Adv. Math. 225 (2010), no. 5, 2464--2516.

\bibitem[HS]{HS} S. Huang and C. Sogge, 
\textit{Concerning $L^p$ resolvent estimates for simply connected manifolds of constant curvature}. J. Funct. Anal. 267 (2014), no. 12, 4635--4666. 

\bibitem[KS]{KS} R. Kunze and E. M. Stein, 
\textit{Uniformly bounded representations and harmonic analysis of the $2\times2$ unimodular group}, Am. J. Math. 82 (1960), p. 1--62.

\bibitem[LP1]{LP1} P. Lax and R. Phillips, 
\textit{The asymptotic distribution of lattice points in Euclidean and non-Euclidean spaces}. J. Funct. Anal. 46 (1982), no. 3, 280--350.

\bibitem[LP2]{LP2} P. Lax and R. Phillips, 
\textit{Translation representation for automorphic solutions of the wave equation in non-Euclidean spaces}. I. 
Comm. Pure Appl. Math. 37 (1984), no. 3, 303--328. 

\bibitem[P]{P} S. Patterson, 
\textit{The limit set of a Fuchsian group}. Acta Math. 136 (1976), no. 3--4, 241--273.  

\bibitem[R]{R} J. V. Ralston, 
\textit{Approximate eigenfunctions of the Laplacian}, J. Differential Geometry 12:1 (1977), 87--100.

\bibitem[So]{So} C. Sogge,
\textit{Fourier integrals in classical analysis}. Second edition. Cambridge University Press, Cambridge, 2017. 


\bibitem[St]{St} E. M. Stein, 
\textit{Oscillatory integrals in Fourier analysis}. Ann. of Math. Stud., 112, 307--355. Princeton Univ. Press, Princeton, NJ, 1986.

\bibitem[Su]{Su} D. Sullivan, 
\textit{The density at infinity of a discrete group of hyperbolic motions}. Inst. Hautes \'Etudes Sci. Publ. Math. No. 50 (1979), 171--202.

\bibitem[T]{T} P. A. Tomas,
\textit{A restriction theorem for the Fourier transform}, Bull. Amer. Math. Soc. 81 (1975), 477--478.

\bibitem[Z]{Z} S. Zelditch,
\textit{Recent developments in mathematical quantum chaos}.Current developments in mathematics, 2009, 115--204, Int. Press, Somerville, MA, 2010. 

\end{thebibliography}
\end{document}